\documentclass{article}[draft]
\usepackage{amsthm, amsfonts, amssymb,latexsym, color}

\title{Incidences of M\"{o}bius transformations in $\mathbb F_p$}
\newcommand{\F}{\mathbb{F}}

\newcommand{\Pb}{\mathbb{P}}
\usepackage{amsmath}
\usepackage[a4paper, total={6in, 8in}]{geometry}
\newtheorem{theorem}{Theorem}
\newtheorem{corollary}{Corollary}
\author{Audie Warren and James Wheeler}

\begin{document}

 \maketitle
 
 \begin{abstract}
We develop the methods used by Rudnev and Wheeler \cite{MishaJames} to prove an incidence theorem between arbitrary sets of M\"{o}bius transformations and point sets in $\mathbb F_p^2$. We also note some asymmetric incidence results, and give applications of these results to various problems in additive combinatorics and discrete geometry.
\end{abstract}

 \section{Introduction} \label{introduction}
In this paper, we study incidence problems concerning M\"{o}bius transformations (which we often simply call \emph{transformations} for brevity) over the prime fields $\mathbb F_p$. Such problems have been studied over the complex numbers, see \cite{solymositardos}. M\"{o}bius transformations over $\F_p$ are bijections $f: \Pb(\F_p) \rightarrow \Pb(\F_p)$ of the form
 $$f(x) = \frac{ax + b}{cx + d}, \quad ad - bc \neq 0$$
 with $a,b,c,d \in \mathbb F_p$, together with the two further points $f([\frac{-d}{c} ; 1]) = [1;0]$ and $f([1;0]) = [\frac{a}{c} ; 1]$. It is convenient to write such transformations as a matrix, via the map
 $$\frac{ax + b}{cx + d} \rightarrow \begin{pmatrix}a & b \\ c & d \end{pmatrix},$$
  where we see that the matrix should be considered \emph{projectively}, since scalar multiples of a matrix define the same transformation. Given a M\"{o}bius transformation $f$, let $M_f$ denote the corresponding matrix. The composition of two M\"{o}bius transformations $f$ and $g$ corresponds to multiplication of the matrices $M_f$ and $M_g$, that is,
 $$M_{f(g(x))} = M_{f}M_{g}.$$

 Given $P \subseteq \mathbb F_p^2$ a set of points, and $C$ a set of curves, we define
$$I(P,C) := \left| \left\{ (p,c) \in P \times C : p \text{ lies on } c \right\} \right|,$$
where each constituent pair $(p,c)$ is called an \emph{incidence}, and $I(P,C)$ is called the number of \emph{incidences} between $P$ and $C$. An \emph{incidence theorem} gives an upper bound for $I(P,C)$. The most common case is where $C$ is a set of lines. In this case, over the real numbers we have the Szemer\'{e}di - Trotter theorem, which states the following\footnote{We use the standard Vinogradov notation $X \ll Y$ to mean that there exists an absolute constant $C$ with $X \leq CY$. We have $Y \gg X$ iff $X \ll Y$. We write $X\sim Y$ to mean that $X \ll Y$ and $Y \ll X$. }.
\begin{theorem}[Szemer\'{e}di - Trotter]For any finite set of points $P$ and lines $L$ in the real plane, we have that
$$I(P,L) \ll |P|^{2/3}|L|^{2/3} + |P| + |L|.$$
\end{theorem}
This theorem has found widespread applications in both additive combinatorics and discrete geometry. Many incidence theorems for sets of curves other than lines have also been proven, see for example \cite{pachsharir}.

Over finite fields, much less is known. Stevens and de Zeeuw \cite{SophieFrank} proved an incidence theorem for points and lines in arbitrary fields, which has found many applications. Incidence theorems for curves other than lines have been infamously elusive in this setting. The first result in this direction was proved by Bourgain \cite{Bourgain}, who gave a qualitative result concerning incidences between Cartesian products and modular hyperbolas. A second incidence theorem concerning modular hyperbolas was proved by Shkredov \cite{shkredovhyperbolas}. The work \cite{MishaJames} is the third instance of a non-linear incidence theorem in $\F_p$. Specifically, the authors in \cite{MishaJames} proved an incidence bound between Cartesian product point sets, and a set of curves given by translates of the hyperbola $xy=-1$.

In this paper, we build upon the ideas present in \cite{MishaJames} to prove incidence bounds between arbitrary sets of M\"{o}bius transformations and point sets in $\F_p^2$. We then exploit these theorems to give applications to problems in additive combinatorics and discrete geometry.

\subsection{Incidence Results}
Given a set of points $P$ and (M\"{o}bius) transformations $T$, we say that a transformation $f$ is $k$-rich (with respect to $P$) if $|f \cap P| \geq k$. Our main theorem is the following.

\begin{theorem} \label{main}
For any set $T$ of M\"{o}bius transformations, and any set of points $P \subseteq \mathbb F_p^2$ with $|P| \leq p^{15/13}$, we have
$$I(P,T) \ll |P|^{15/19}|T|^{15/19} + |P|^{23/19}|T|^{4/19} + |T|.$$
Furthermore, given any set $P$ of points with $|P| \leq p^{15/26}$ and some integer $k \geq 3$, the set $T_k$ of $k$-rich transformations satisfies
$$|T_k| \ll \frac{|P|^{15/4}}{k^{19/4}} + \frac{|P|^2}{k^2}.$$
\end{theorem}
The second term in the incidence bound above should be seen as an error term accounting for when the point set is proportionally very large. Usually we would expect this error term to be $|P|$, however as an artefact of the method used we recover a larger error term. Note that in contrast to many other incidence theorems, in the balanced case $|P| =|T| = N$, the main term has an exponent of $30/19 > 3/2$. This is in contrast to, say, line incidences, which would correspond to a trivial (via Cauchy-Schwarz) upper bound of $N^{3/2}$. The difference is that M\"{o}bius transformations require \emph{three} points to be uniquely defined. This implies that the trivial bound implied by H\"{o}lder's inequality is $N^{5/3}$, showing that Theorem \ref{main} is non-trivial in the balanced case.

We also note that the same proof, applied with an alternative incidence theorem of Stevens and de Zeeuw, gives the following result concerning asymmetric Cartesian products. The symmetric form where $A=B$ is essentially Theorem 3.2 in \cite{MishaJames}. 

\begin{theorem}\label{asymmetric}
Let $A \times B$ be a set of points in $\mathbb F_p^2$, and let $T$ be any set of M\"{o}bius transformations. Then if $|A||T| \ll p^2$, we have
$$I(A \times B, T) \ll |A|^{4/5}|B|^{3/5}|T|^{4/5} + |A|^{6/5}|B|^{7/5} |T|^{1/5} + |T|$$
Furthermore, given any set $A \times B$ of points with $|A|^3|B|^2 \leq p^{2}$ and some integer $k \geq 3$, the set $T_k$ of $k$-rich transformations satisfies
$$|T_k| \ll \frac{|A|^4|B|^3}{k^{5}} + \frac{|A|^2|B|^2}{k^2} .$$
\end{theorem}

Note that this theorem is non-trivial in the range $|T| \geq \frac{|A|^4|B|^3}{\max \{|A|,|B|\}^5}$, and is precisely trivial in the balanced case $|A| = |B| = N$, $|T| = N^2$, since any transformation can only have at most $\max\{|A|,|B|\}$ incidences.

Using an argument originating in a paper of Rudnev and Shkredov \cite{rudnevshkredov} and Theorem \ref{asymmetric} as a vehicle, the work of Rudnev and Wheeler can be adapted to prove the following result, which can be seen as an analogue of \cite[Theorem 8]{rudnevshkredov}. 

\begin{theorem}\label{energyversion}
Let $A \times B$ be a set of points in $\F_p^2$ with $|B| \leq p^{1/2}$, and $T$ be a set of M\"{o}bius transformations. Then we have
$$I(A \times B, T) \ll |A|^{1/2}|B|^{7/10}|T|^{3/5}E(T)^{1/10} + |B|^{1/2}|T|$$
where
$$E(T) := |\{ (f_1,f_2,f_3,f_4) \in T^4 : f_1 f_2^{-1} = f_3 f_4^{-1}\}|.$$
\end{theorem}
The quantity $E(T)$ is termed the \emph{energy} of $T$. Such energies have been studied in the case where $T$ is a set of lines, see for example \cite{affinegroup}.

Rudnev and Wheeler \cite{MishaJames} specifically studied the case where $T$ is a set of translates (which we now rename as $H$) of the modular hyperbola $xy=\pm1$, and proved energy bounds for such sets. Translates of $xy=1$ are equations of the form
$$(y-a)(x-b) = 1$$
which can be rearranged to the form of a M\"{o}bius transformation. Applying the energy bound \cite[Lemma 5.2]{MishaJames} for $E(H)$ (explicitly $E(H)\ll|H|^2M$ with $M$ as defined below), Theorem \ref{energyversion} implies the following.

\begin{theorem}\label{asymmetrichyperbola} 
Let $A \times B \subseteq \mathbb F_p^2$ be a set of points with $|B| \leq p^{1/2}$, and let $H$ be a set of translates of the hyperbola $xy=\pm1$. Then we have
$$I(A \times B, H) \ll |A|^{1/2}|B|^{7/10}|H|^{4/5}M^{1/10} + |B|^{1/2}|H|$$
where $M$ is the maximum number of translates in $H$ having the same $x$-translate or $y$-translate.
\end{theorem}

 Remarking on $M$, most obviously it is bounded above by $M\leqslant |H|$ and so is an improvement on the trivial energy bound for all but the very worst case when all translates share an abscissa or ordinate and thus are all on a horizontal or vertical line. At the opposite extreme when $M=1$ the energy bound of the hyperbola is minimal at $|H|^2$, this occurs when each translate maps to a different horizontal and vertical line. When $|H|>|A|^{3/2}$, \cite{MishaJames} instead uses a pruning argument to improve on this $M$, instead allowing it to be replaced with a factor of $|H|^{2/11}|A|^{8/11}$ in cases where this is an improvement. 

\subsection{Applications}

 Given sets $A,B \subseteq \F_p$, we define the \emph{sumset} and \emph{product set} of $A$ and $B$ as 
$$A+B = \{ a+b : a \in A, b \in B \}, \qquad AB = \{ab:a \in A, b \in B \}.$$
Furthermore, we define the representation function
$$r_{AB}(x) = \{ (a,b) \in A \times B : ab = x \}.$$
As a first application, we improve a result of Shkredov \cite{shkredovhyperbolas} regarding the number of representations of a non-zero product $\lambda \in AA$, given $A$ has small sum-set. Specifically, Theorem \ref{asymmetrichyperbola} implies the following more general result.

\begin{corollary} \label{representations}
Let $A,B \subseteq \F_p^*$ with $|A| = |B| = N$ satisfy $|A + B| \leq KN \leq p^{1/2}$. Then for all non-zero $\lambda \in AB$, we have
$$r_{AB}(\lambda) \ll K^{6/5}N^{9/10}.$$
\end{corollary}
The condition $|A| = |B|$ is not strictly necessary above, as the proof can be followed with arbitrary sizes of $A$ and $B$. For simplicity, we have stated the result in the balanced case.

As a second application, we prove an equivalent of Beck's Theorem for M\"obius transformations. That is, we show that any point set in $\F_p^2$ lies in one of two extremes: either many points lie on a single transform, or the set defines many transforms, where we say a transformation is defined by a point set if it passes through three of its points. 
\begin{corollary}\label{cor:Becks}
Let $P \subseteq \F_p^2$ be an arbitrary point set, with $|P| \leq p^{15/13}$. Then there exist positive constants $C$ and $K$ such one of the following two statements is true:
\begin{itemize}
    \item There is a M\"obius transformation containing at least $\frac{|P|}{C}$ of the points.

\item The point set $P$ defines at least $\frac{|P|^{12/7}}{K}$ M\"{o}bius transformations.
\end{itemize}
\end{corollary}
This follows from a standard proof of Beck's Theorem, replacing the Sz\'emeredi-Trotter bound (or equivalently the Stevens - de Zeeuw incidence bound in finite fields) for rich lines with Theorem \ref{main}, and making the required adjustments for the fact that our transforms are defined by \emph{three} points.

A common use of incidence results is to prove so-called `expander' results. Given a set $A \subseteq \F_p$, we can form, for instance, the sets 
$$AA+A = \{ab+c : a,b,c \in A \} \qquad A(A+1) = \{a(b+1):a,b \in A\}.$$
Both of these sets are expected to be significantly larger than the input set $A$ (subject to a necessary condition that $|A|$ is small with respect to the characteristic $p$), and such claims can be proved very simply from linear incidence theorems. Because of this growth phenomenon, these sets are termed `expanders'. For a nice introduction to this topic in finite fields, see \cite{expandersfinitefields}. Theorem \ref{asymmetric} and Theorem  \ref{asymmetrichyperbola} can be naturally applied to prove the following expander results.

\begin{corollary}\label{expander}
For all $A \subseteq \F_p$ with $|A|\leq p^{1/2}$, we have
$$\left| \left\{ a + \frac{1}{b-c} : a,b,c \in A \right\} \right| \gg |A|^{6/5},$$
$$\left| \left\{ \frac{ab+c}{b+d} : a,b,c,d \in A \right\} \right| \gg |A|^{4/3}.$$
\end{corollary}
Note that the best we could expect for the first expander is an exponent of two, which can be seen by taking $A$ with $|A+A| \ll |A|$.

As a final application, we consider a problem in the family of 'similar configurations', as considered by, for example, Aiger and Sharir \cite{aigersharir}. Such questions ask for the maximum possible number of realisations of a certain configuration inside a larger set. These configurations are usually considered up to some equivalence, e.g. rigid transformations. As an example, taking the configuration to be two points of unit distance in the  real plane, and the equivalence to be rigid transformations, we arrive at the the unit distance problem.

Our choice of `equivalence' will in fact be up to projective transformations. The reason for this is that the set of M\"{o}bius transformations as defined above is equivalent to the set of projective transformations on $\Pb(\F_p)$. Indeed, both are given by $2 \times 2$ matrices of non-zero determinant, up to scalar multiplication. Using this, we can prove the following.

\begin{corollary} \label{similarity} 
Let $A,S \subseteq \mathbb F_p$ with $|S|^3|A|^2\leq p^2$. Then the number of subsets $A' \subseteq A$ which are projectively equivalent to $S$ is $O\left(\frac{|A|^3}{|S|}\right)$.
\end{corollary}

In the above, we say that two sets $A$ and $B$ in $\mathbb F_p$ are projectively equivalent if there exists a projective transformation $\pi$ such that $\pi(A) = B$.

\section{Proof of Theorem \ref{main}}

Before beginning the proof, we explain the main ideas. The aim is to count the number of $k$-rich transformations passing through each $q \in P$. For each fixed $q$, the transformations through $q$ can be mapped to a set of lines, and by adjusting the point set, this map (almost) completely preserves incidences. We can then apply the incidence theorem of Stevens and de Zeeuw \cite{SophieFrank}, obtaining a bound for the $k$-rich transformations through $q$. Summing this bound over each $q$, we count each $k$-rich transformation at least $k$ times. This observation gives us a bound on the number of $k$-rich transformations with respect to the whole point set, proving the second part of the theorem. The first part then follows from the second after a standard dyadic summation.
\begin{proof}
We begin by fixing a point $q = (q_1,q_2) \in P$. Let $T_q$ denote the transformations in $T$ that pass through $q$, and further let $T_{q,k}$ be the $k$-rich transformations in $T_q$. Our goal will be to bound the size of $T_{q,k}$. To do this, we will make use of the following corollary of an incidence result of Stevens and de Zeeuw \cite{SophieFrank} regarding $k$-rich lines.

\begin{corollary}\label{k-rich}
Let $P$ be a set of points in $\mathbb F_p^2$ such that $|P| \leq p^{15/13}$, and let $L_k$ be the set of lines which pass through at least $k \geq 2$ points of $P$. Then we have
$$|L_k| \ll \frac{|P|^\frac{11}{4}}{k^\frac{15}{4}} + \frac{|P|}{k}.$$
\end{corollary}
Any transformation $f \in T_q$ must satisfy $f(q_1) = q_2$. Writing $f$ as
$$f(x) = \frac{ax + b}{cx+d}$$
we have two cases depending on the value of $c$. If $c=0$, then we see that $f(x)$ is actually a line in $\mathbb F_p^2$. At this stage, we do nothing with these transformations. Now assume that $c$ is non-zero, allowing us to scale such that $c=1$. We can now calculate the specific form of $b$, as 
$$q_2 = \frac{aq_1 + b}{q_1 + d} \implies b = q_2(q_1 + d) - aq_1,$$
so that we have
$$M_f = \begin{pmatrix}a & q_2(q_1 + d) - aq_1 \\ 1 & d \end{pmatrix}.$$
Consider the map
$$\phi_q(f) = g_1 ( f (g_2 (x)))$$
for $f \in T_q$, and the functions $g_1$ and $g_2$ being defined as 
$$g_1(x) = \frac{1}{q_2 - x}, \qquad g_2(x) = q_1 - \frac{1}{x}.$$ The idea is that the image of $T_q$ under this map will be a set of \emph{lines} in $\mathbb F_p^2$, where we can apply known incidence theorems. This also requires the preservation of incidences, which will be ensured via an alteration to the point set $P$. To show that $\phi_q(f)$ is indeed a line, we multiply the corresponding matrices.
$$M_{g_1}M_fM_{g_2} = \begin{pmatrix}0 & 1  \\ -1 & q_2 \end{pmatrix}\begin{pmatrix}a & q_2(q_1 + d) - aq_1 \\ 1 & d \end{pmatrix}\begin{pmatrix}q_1 & -1 \\ 1 & 0 \end{pmatrix} = \begin{pmatrix}q_1 + d & -1 \\ 0 & a - q_2 \end{pmatrix}.$$
Note that since $\det(M_{g_1}) = \det(M_{g_2}) = 1$, we have $\det(M_{g_1}M_fM_{g_2}) = \det(M_f) \neq 0$.
Since the bottom left entry is zero, this composition gives the line
$$y = \frac{(q_1 + d)x -1}{a - q_2}.$$
Furthermore, we see that distinct transformations $f \in T_q$ map to distinct lines, i.e. $\phi_q$ is injective, and moreover we cannot have $a = q_2$ as this would imply $\det(M_f) = 0$. Indeed, suppose that for some pair of transformations 
$$f(x) = \frac{ax+(q_2(q_1 + d) - aq_1)}{x+d}, \quad f'(x) =\frac{a'x+(q_2(q_1 + d') - a'q_1)}{x+d'},$$ we have $\phi(f) = \phi(f')$. Comparing the intercepts, we must have $\frac{-1}{a-q_2} = \frac{-1}{a'-q_2} \implies a = a'$, and then by comparing the slopes we find $d = d'$, so that $f = f'$. To prove the second claim, we see that if $a = q_2$, then the determinant of $M_f$ is 
$$\det(M_f) = ad - q_2(q_1+d) - aq_1 = ad - aq_1 - ad - aq_1 = 0.$$

Having mapped the M\"{o}bius transformations to lines, we must now alter the point set $P$ to preserve incidences. We define the map
$$\psi(s_1,s_2) = \left( g_2^{-1}(s_1) , g_1(s_2)\right) = \left( \frac{1}{q_1 - s_1}, \frac{1}{q_2 - s_2}\right)$$
where we make the relevant restrictions $s_1 \neq q_1$ and $s_2 \neq q_2$. When applied to $P$, this gives the point set
$$ P':= \psi(P) = \left\{ \left( \frac{1}{q_1 - s_1}, \frac{1}{q_2 - s_2}\right) : (s_1,s_2)\in P \right\}.$$
We claim that the map $\psi$ satisfies the property that for all $s \in P$ not lying on the lines $x = q_1$ or $y = q_2$, and $f \in T_q$, we have that $s$ lies on $f$ if and only if $\psi(s)$ lies on $\phi(f)$. Indeed, we have
$$s_2 = f(s_1) \iff g_1(s_2) = g_1(f(s_1)) \iff g_1(s_2) = g_1(f(g_2(g_2^{-1}(s_1))))$$
and this final equality is precisely the statement that $\psi(s)$ lies on $\phi(f)$. Finally, the restrictions $s_1 \neq q_1$ and $s_2 \neq q_2$ removes only the single incidence where $(q_1,q_2) = (s_1,s_2)$. Indeed, this follows from $f$ being a bijection. This implies that $k$-rich transformations with respect to $P$ have been mapped to $k-1$ rich lines with respect to $P'$. This lets us bound the quantity $T_{q,k}$ for $k \geq 3$ via Corollary \ref{k-rich}, as
$$|T_{q,k}| \ll \frac{|P'|^{11/4}}{(k-1)^{15/4}} + \frac{|P'|}{k-1} \ll \frac{|P|^{11/4}}{k^{15/4}} + \frac{|P|}{k}.$$
Furthermore, the transformations in $T_{q,k}$ with $c=0$, which are actually lines, must be concurrent through $q$, and the number of such $k$-rich lines is at most $\frac{|P|}{k}$. This is then absorbed into the above bound.

In order to convert this bound for $T_{q,k}$ into a bound on the set $T_k$ of $k$-rich transformations in $T$, we sum over each point $q \in P$, for each point counting the $k$-rich transformations passing through $q$. Note that this necessarily counts each $k$-rich transformation at least $k$ times. Performing this summation, we find
$$|T_k| \leq \frac{1}{k} \sum_{q \in P}|T_{q,k}| \ll \frac{|P|^{15/4}}{k^{19/4}} + \frac{|P|^2}{k^2}.$$

We now complete the proof of Theorem \ref{main} in a standard manner. We perform a dyadic decomposition, and split the sum in terms of some parameter $\Delta$ to be chosen later. In the following, we let $T_{=k}$ denote the set of exactly $k$-rich transformations in $T$.
\begin{align*}I(P,T) & = \sum_{k=1}^{|P|} k |T_{=k}| \\ & = \sum_{k < \Delta} k |T_{=k}| + \sum_{k=\Delta}^{|P|} k |T_{=k}| \\ & \ll \Delta |T| + \sum_{i=1}^{\log|P|} \sum_{\substack{f \in T: \\  2^i \Delta \leq |f \cap P| <2^{i+1}\Delta}} (2^{i+1}\Delta) \\ & \leq \Delta |T| + \sum_{i=1}^{\log|P|} |T_{2^i\Delta}|(2^{i+1}\Delta) \\
& \ll \Delta |T| + \sum_{i=1}^{\log|P|}\left( \frac{|P|^{15/4}}{(2^{i+1}\Delta)^{19/4}} + \frac{|P|^2}{(2^{i+1}\Delta)^2} \right)(2^{i+1}\Delta) \\
& \ll \Delta |T| + \frac{|P|^{15/4}}{\Delta^{15/14}} + \frac{|P|^2}{\Delta}\end{align*}

We now choose $\Delta$ in order to optimise the first two terms. This is achieved when
$$\Delta |T| \sim \frac{|P|^{15/4}}{\Delta^{15/14}} \implies \Delta \sim \frac{|P|^{15/19}}{|T|^{4/19}}.$$
We now recall that the bound applied for $T_{k}$ only applies for $k \geq 3$, therefore we must impose the restriction $\Delta \geq 3$. Our choice for $\Delta$ is then
$$\Delta = \max \left\{ 3, \frac{|P|^{15/19}}{|T|^{4/19}} \right\}.$$
In the case $\Delta =3$, we must have $$\frac{|P|^{15/19}}{|T|^{4/19}} \leq 3 \implies |P|^{15/4} \ll |T|,$$
and thus the bound above gives 
$$I(P,T) \ll |T| + |P|^{15/4} + |P|^{2} \ll |T|.$$
In the second case where we choose $\Delta = \frac{|P|^{15/19}}{|T|^{4/19}}$, we find the bound
$$I(P,T) \ll |P|^{15/19}|T|^{15/19} + |P|^{23/19}|T|^{4/19}.$$
Putting these two bounds together, we conclude that
$$I(P,T) \ll |P|^{15/19}|T|^{15/19} + |P|^{23/19}|T|^{4/19} + |T|$$
as needed.
\end{proof}

\section{Applications}
In this section we prove Corollaries 1-4. We begin by proving Corollary \ref{representations}.
\begin{proof}[Proof of Corollary \ref{representations}]
Without loss of generality, we may assume $\lambda = 1$. Given a pair $(a,b) \in A \times B$ such that $ab=1$, we have that for any pair $(a',b') \in A \times B$,
\begin{align*}
    1 &= ab \\
    & = (a + b' - b')(b + a' - a')\\
    & = (Y - c)(X - d)
\end{align*}
which is an incidence between the point set $P = (A+B)^2$ and the $N^2$ hyperbolas $H$ given by $(Y-b')(X-a') = 1$ as above. Applying Theorem \ref{asymmetrichyperbola}, we find
$$r_{AB}(1) N^2 \leq I(P,H) \ll |A+B|^{6/5}N^{17/10} + |A+B|^{1/2}N^2.$$

If the second term dominates, we have that $r_{AB}(1) \ll K^{1/2}N^{1/2}$, which is better than claimed. We may therefore assume that the leading term dominates, giving
$$r_{AB}(1) \ll K^{6/5}N^{9/10}$$
as needed.
\end{proof}

Secondly, we prove Corollary~\ref{cor:Becks}.
\begin{proof}[Proof of Corollary \ref{cor:Becks}]

Given our point set $P$ of size $n$, we will consider all the M\"obius transformations defined by $P$, where we recall that a M\"obius transformation is said to be defined by $P$ if it passes three points of $P$. For each $k \geq 0$, we define a transformation $f$ to be $2^k$-rich if \emph{between} $2^k$ and $2^{k+1}-1$ points of $P$ lie on the transform (note that this is different to the definition used previously). Specifically, $f$ is $2^k$-rich if $2^k \leq |f \cap P| < 2^{k+1}$. The outline of the proof will be to define a central range of rich but not too rich transforms, and show that whilst they cover a large number of the points, they miss a positive proportion of them. We then have two cases: either we have a positive proportion of $P$ lying on richer transforms, which will lead to the first conclusion, or on less rich transforms, in which we must have many transforms to support all of the points of $P$. Having defined $2^k$-rich transforms, we note that by Theorem~\ref{main}, we must have only $$O\left(\frac{n^{\frac{15}{4}}}{2^{\frac{19j}{4}}}+\frac{n^2}{2^{2j}}\right)$$ such $2^k$-rich transforms. Each of these $2^k$-rich transforms contain $\Omega(2^{3j})$ triples of points of $P$, and so at most $$O\left(\frac{n^{\frac{15}{4}}}{2^{{\frac{7j}{4}}}}+n^2 2^{j}\right)$$ triples of points are on a $2^j$-rich transform. Let $C$ be a large constant, and consider the transforms which are $2^j$-rich for 
$$Cn^{3/7}\leq 2^{j}\leq n/C^{7/4}.$$
Summing the number of triples of points on all such transforms, we have at most $O\left(\frac{n^{3}}{C^{{7/4}}}\right)$ triples of points on this collection of rich but not too rich transforms. This can be compared to the total number of all triples of points $\frac{n(n-1)(n-2)}{6}$ to see that by taking $C$ sufficiently large, a positive proportion of triples do not lie on a rich transform in our above defined collection. This means that a positive proportion of triples lie on transforms which are either richer than our collection, or poorer. This will lead to the two cases in our corollary.

Explicitly, either a positive proportion of triples lie on transforms with less than $2Cn^{3/7}$ points, or more than $\frac{n}{C^{7/4}}$ points. If we are in the latter case, we are done as this is the first conclusion. We will then assume we are in the former case, and none of these transforms have more than $2Cn^{3/7}$ points.

Let $T$ be this set of poor transformations. Counting the number of triples on these poor transforms gives
$$n^{9/7}|T| \gg \sum_{f \in T} |f \cap P|^3 \gg n^3 \implies |T| \gg n^{12/7}$$
proving the second case of the corollary.
\end{proof}

Next, we prove Corollary \ref{expander}.

\begin{proof}[Proof of Corollary \ref{expander}]
Recall that we are considering the set
$$ Q := \left\{ a + \frac{1}{b-c} : a,b,c \in A \right\}.$$
Consider the point set $Q \times A$, and the set of hyperbolas
$$H := \{(y-a)(x-c) = 1:a,b \in A\}.$$
The main observation is that for all $a,b,c \in A$, we have
$$q = a + \frac{1}{b-c} \iff (b-c)(q-a) = 1$$
that is, the point $(q,b)$ lies on the hyperbola $(y-c)(x-b) = 1$. This shows that $I(Q \times A,H) \geq |A|^3$. Applying the upper bound given by Theorem \ref{asymmetrichyperbola}, we have
$$|A|^3 \leq I(Q \times A, H) \ll |Q|^{1/2}|A|^{12/5} + |A|^{5/2}.$$
Discarding the error term and rearranging yields the result
$$|Q| \gg |A|^{6/5}.$$

The second expander follows in a similar way, by considering the set of transformations $T$ given by
$$y = \frac{ax + c}{x + d}$$
with $a,c,d \in A$, $ad - c \neq 0$, and the point set $A \times Q$, where we define
$$Q := \left\{ \frac{ab+c}{b+d} : a,b,c,d \in A \right\}.$$
Note that we have $|T| \gg |A|^3$, since there are only at most $|A|^2$ bad triples $(a,c,d)$ with $ad-c = 0$. We can upper and lower bound the incidences between these points and transforms by Theorem \ref{asymmetric}, giving
$$|A|^4 \leq I(A \times Q,T) \ll |A|^{16/5} |Q|^{3/5} + |A|^{9/5}|Q|^{7/5} + |A|^3.$$
The third term can be discarded, and the second term leads to a result better than claimed. Given that the leading term dominates, rearranging gives the result as needed.
\end{proof}
Finally, we prove Corollary \ref{similarity}.
\begin{proof}[Proof of Corollary \ref{similarity}]
Recall that we are given two sets $A,S \subseteq \F_p$, which should be thought of as the `large set' and the `pattern set' respectively. We are aiming to upper bound the number of subsets $A' \subseteq A$ which are projectively equivalent to $S$. In other words, we are looking for the number of projective transformations $f$ such that $f(S) \subseteq A$. Let $T$ be the set of such transformations. Such projective transformations take the form of a $2 \times 2$ matrix with non-zero determinant, up to scalar multiplication. In other words, they are given by M\"{o}bius transformations.

For context, each triple of points in the plane $\F_p^2$ has at most one M\"{o}bius transformation passing through all three points. The result of this is that upon sending three points of $S$ to three points to $A$, the M\"{o}bius transformation is determined, giving a trivial upper bound of $|T| \ll |A|^3$.

Suppose we have a transformation $f \in T$. Since we must have $f(S) \subseteq A$, when we consider the point set $S \times A$, $f$ must be $|S|$ rich with respect to this point set. Using the upper bound for $k$-rich transformations from Theorem \ref{asymmetric}, we have
$$|T| \ll \frac{|S|^4|A|^3}{|S|^5} + \frac{|S|^2|A|^2}{|S|^2} = \frac{|A|^3}{|S|}  + |A|^2.$$
Finally, we note that we must have $|S| \leq |A|$, so that the error term may be ignored.
\end{proof}

\section*{Acknowledgements}
The first author was supported by Austrian Science Fund FWF grant P-34180. We thank Oliver Roche-Newton, Misha Rudnev, and Sophie Stevens for helpful suggestions and conversations.

\end{document}